\documentclass[10pt]{article}
\usepackage{amsmath,amsfonts,euscript,amssymb,amsthm,graphicx,verbatim, esint}
\usepackage{graphicx}
\usepackage{color}
\usepackage{hyperref}
\usepackage{enumerate}
\usepackage{mathtools}
\usepackage{caption}
\usepackage{subcaption}
\usepackage{authblk}

\newtheorem{theorem}{Theorem}

\newtheorem{lemma}{Lemma}
\newtheorem{proposition}{Proposition}
\newtheorem{remark}{Remark}

\newcommand{\R}{\mathbb{R}}
\newcommand{\N}{\mathbb{N}}
\newcommand{\Z}{\mathbb{Z}}

\newcommand*\lap{\Delta}

\newcommand{\eps}{\varepsilon}

\newcommand{\dif}{\ensuremath{\,\mathrm{d}}}

\newcommand{\BMO}{\mathrm{BMO}}
\newcommand{\VMO}{\mathrm{VMO}}

\DeclarePairedDelimiter{\abs}{\lvert}{\rvert} 
\DeclarePairedDelimiter{\norm}{\lVert}{\rVert} 
\DeclarePairedDelimiter{\bra}{(}{)} 
\DeclarePairedDelimiter{\pra}{[}{]} 
\DeclarePairedDelimiter{\set}{\{}{\}} 

\title{Well-posedness and stability for a class of fourth-order nonlinear parabolic equations}

\author{Xinye Li \qquad Christof Melcher}

\AtEndDocument{%
 \par
 \medskip
 \begin{tabular}{@{}l@{}}%
 Xinye Li, School of Mathematics and Statistics, Central South University, \\410083 Changsha, China\\
 \textit{Email:} \texttt{xinye.li@csu.edu.cn}\\
 \\
Christof Melcher, Lehrstuhl f\"ur Angewandte Analysis, RWTH Aachen University,\\ 52056 Aachen, Germany \\
 \textit{Email:} \texttt{melcher@math1.rwth-aachen.de}\\
 \end{tabular}}
 
\begin{document}

\maketitle
\begin{abstract}
In this paper we examine well-posedness for a class of fourth-order nonlinear parabolic equation $\partial_t u + (-\lap)^2 u = \nabla \cdot F(\nabla u)$, where $F$ satisfies a cubic growth conditions. We establish existence and uniqueness of the solution for small initial data in local BMO spaces. In the cubic case  $F(\xi) = \pm \abs{\xi}^2 \xi$, we also examine the large time behavior and stability of global solutions for arbitrary and small initial data in $\VMO$, respectively. 



\end{abstract}

%
%

\section{Introduction and main results}

We consider fourth-order parabolic equations of the form
\begin{equation}\label{eq:main}
\partial_t u + (-\lap)^2 u = \nabla \cdot F(\nabla u)
\end{equation}
where $F:\R^n \to \R^n$ is a $C^1$ nonlinearity. Evolution equations of this type have been proposed in connection with the epitaxial growth of thin films, see \cite{ortiz, schulze1999geometric, zangwill}. 
In this work we consider the following growth conditions on the nonlinearity:
the derivative $F':\R^n \to \R^n$ satisfies a Lipschitz condition of the form
\begin{equation}\label{eq:lip}
\abs*{ F'(\xi_1)-F'(\xi_2)} \le C\bra*{\abs*{\xi_1} + \abs*{\xi_2} } \abs*{\xi_1 - \xi_2} \quad\text{for all}\quad \xi_1, \xi_2 \in \R^n
\end{equation}
where $C>0$. Cubic nonlinearities $F(\xi)=\pm \abs{\xi}^2 \xi$ are the prototype.
If $F(\xi)=W'(\xi)$ for a $C^2$ potential $W:\R^n \to \R$, then \eqref{eq:main} is the $L^2$-gradient flow equation for the energy
\begin{equation} \label{eq:energy}
E(u)=\int \frac{1}{2} \abs{\nabla \nabla u}^2 + W(\nabla u) \dif x.
\end{equation}
For cubic nonlinearities we have $W(\xi)=\pm \frac{1}{4}\abs{\xi}^4$. As the energy formally decreases in time, the sign lets us distinguish the coercive and non-coercive case, respectively.

The questions of interest pertains to (i) the well-posedness of \eqref{eq:main} with $F$ satisfying \eqref{eq:lip} for initial data in a certain function space, and (ii) the asymptotic property and stability of global solutions. Drawing upon a variational setting, King et al. \cite{king} established the existence, uniqueness, and regularity of solutions in bounded domains with Neumann boundary conditions. This was achieved under certain monotonicity and growth assumptions on the nonlinearity $F$ (e.g., $F(\xi)=|\xi|^{p-2}\xi - \xi$ with $p>2$) for initial data in $H^2 \cap W^{1,p}$ by proving an energy dissipation property. They also discussed the stability of solutions in a one-dimensional setting. One of us demonstrated local and global well-posedness for initial data in homogeneous Sobolev spaces $\dot{W}^{1,\frac{\alpha n}{2}}(\R^n)$ for nonlinearities $F$ satisfying a $\alpha$-Lipschitz condition 
(e.g., $F(\xi)= \pm |\xi|^{\alpha}\xi$) similar to \eqref{eq:lip} for $\alpha>2/n$, see \cite{melcher}. 
Here and in what follows homogeneous Sobolev spaces $\dot{W}^{1,q}(\R^n)$ are identified with the completion of $C^\infty_0(\R^n)$ with respect to the $W^{1,q}$ seminorm, see \cite{Brasco2021}. Sandjo et al. \cite{sandjo2015} studied \eqref{eq:main} on bounded smooth domains with appropriate boundary conditions and for nonlinearities $F$ satisfying an $\alpha$-Lipschitz condition as in \cite{melcher} but for a different range of $\alpha \in (1,2)$ and developed a local and global $L^\frac{\alpha n}{2-\alpha}$ theory. Ishige et al. \cite{ishige2020} proved \eqref{eq:main}, for the power nonlinearity $F(\xi)=-\abs{\xi}^{p-2} \xi$ with $p>2$, local and global-in-time existence results in terms of weak local Lebesgue bounds
on initial functions or gradients, respectively, depending on the size of $p$. They also gave a sufficient condition for the maximal existence time of the solution to be finite. The existence results were enhanced in \cite{ishige2022} by the use of 
local Morrey spaces in the framework of a systematic study on the solvability of the Cauchy problem for higher-order nonlinear parabolic equations. Feng et al. \cite{feng2022} proved well-posedness of \eqref{eq:main} in the spatially periodic case for the same nonlinearity $F(\xi)=-\abs{\xi}^{p-2} \xi$ for $2<p<3$ with and without advection term of an incompressible vector field. 

For cubic nonlinearities $F(\xi)=\pm \abs{\xi}^2 \xi$ the equation \eqref{eq:main} becomes scaling invariant, i.e., a smooth solution $u:\R^n \times (0,\infty) \to \R$ of \eqref{eq:main} with initial data $a$ induces a one-parameter family
\begin{equation}\label{eq:scaling}
u_\lambda (x,t)= u(\lambda x, \lambda^4 t), \quad \lambda>0
\end{equation}
of solutions of \eqref{eq:main} for the initial data
\[
a_\lambda (x) = a(\lambda x).
\]
Concerning well-posedness, this motivates a framework of function spaces based on semi-norms that feature the same kind of scaling invariance. 
An obvious example is the homogeneous Sobolev space $\dot{W}^{1,n}(\R^n)$ for which small data global well-posedness and large data local well-posedness has been shown in \cite{melcher} under the assumption of \eqref{eq:lip} if $n<4$. 
Finite time blow-up and refined existence results in the non-coercive case are shown in \cite{ishige2020} and \cite{ishige2022}, where only smallness condition of the initial gradient in local or global weak Lebesgue spaces $L^n_{\rm weak}(\R^n)$ and Morrey spaces $\mathcal{M}_{1,n-1}(\R^n)$
\footnote{Here we adopt the following convention for the (global) Morrey norm (cf. e.g. \cite{Giusti}): \[
\norm{a}_{\mathcal{M}_{p,\lambda}} =\bra*{  \sup_{\substack{x \in \R^n\\ r >0}}   \frac{1}{r^\lambda}  \int_{B(x,r)} | a(y)|^p \dif y}^{\frac{1}{p}}.
\]}, respectively, are required.

The coercive case $F(\xi)= \abs{\xi}^2 \xi$ provides in addition a uniform control of $\nabla u$ in $L^4(\R^n)$ which is independent of the dimension $n$. In combination with the 
a priori estimate
\[
\frac{d}{dt} \| \nabla^k u\|_{L^2}^2 +  \| \nabla^{k+2} u\|_{L^2}^2 \lesssim  \|\nabla u\|_{L^4}^2 \| \nabla^{k+1} u\|_{L^4}^2 
\]
for all $k \in \N$ and interpolating $\| \nabla^{k+1} u\|_{L^4}$ between $\|\nabla u\|_{L^4}$ and  $\| \nabla^{k+2} u\|_{L^2}$ one deduces global existence and regularity 
in dimension $n \le 3$. In the critical dimension $n=4$, the Sobolev inequality  $\| \nabla^{k+1} u\|_{L^4} \lesssim \| \nabla^{k+2} u\|_{L^2}$ provides uniform global bounds under 
a smallness condition on $\|\nabla u\|_{L^4}^2$. For a local solvability result without a size restriction on the initial data, see \cite{li2017GWP}. In dimension $n \ge 5$ certain regularity criterion of Serrin type have been proposed in \cite{fan2017GWP}. \\

Aiming for a somewhat unifying framework, we investigate \eqref{eq:main} in homogeneous spaces of functions of bounded mean oscillation ($\BMO$ and $\VMO$), which are good candidates for the largest possible scaling invariant spaces containing $\dot{W}^{1,n}(\R^n)$ where well-posedness of \eqref{eq:main} can be proven. The idea of BMO well-posedness of parabolic PDEs was introduced by Koch and Tataru in their seminal work on the Navier-Stokes equation \cite{koch2001}. Since then, there have been many interesting applies, e.g. for geometric flows such as the heat flow of biharmonic and polyharmonic maps, and for a surface growth model, see \cite{blomker2012, huang2011, kochlamm, wang2012}. 

The concept of BMO well-posedness rests on the Carleson-type characterisation of BMO functions based on space-time extension and averaging, see \cite{stein}.
The natural extension operator in our case is the semigroup generated by the bi-Laplacian $S(t)=e^{-t\lap^2}$ as, see Lemma \ref{la:BMO}.  This induces the following
notion of local $\BMO$ semi norms for $a\in L_{\rm loc}^1(\R^n)$ 
\[
\norm{a}_{\BMO_R} := \sum_{k=1}^2 \sup_{\substack{x \in \R^n\\ 0<r < R}} 
\bra*{  \int_0^{r^4} \fint_{B(x,r)} t^{\frac{2k-4}{4}}|\nabla^k S(t)a(y)|^2 \dif y \dif t }^{\frac{1}{2}} 
\]
defining local $\BMO$ spaces $\BMO_R$ for all $R>0$. For $R=\infty$, the space $\BMO_{\infty}$ is identical to the classical $\BMO$ space, see Remark \ref{rem:bmo},
characterized by the semi norm
\[
\norm{a}_{\BMO} = \sup_{\substack{x \in \R^n\\ r>0}}  \fint_{B(x,r)} | a(y) -a_{x,r}| \dif y
\]
where 
\[
a_{x,r} := \fint_{B(x,r)} a(y) \dif y := \frac{1}{|B(x,r)|} \int_{B(x,r)} a(y) \dif y 
\] denotes the average over $B(x,r)$. 
$\BMO$ strictly includes the homogeneous spaces of functions with 
gradients in $L^n(\R^n)\subset L^n_{\rm weak}(\R^n) \subset \mathcal{M}_{1,n-1}(\R^n)$, which can be deduced from Poincar\'e's and Jensen's inequality 
\begin{equation}\label{eq:poincare}
 \fint_{B(x,r)} | a(y) -a_{x,r}| \lesssim  r  \fint_{B(x,r)} |\nabla a(y)| \dif y  \lesssim  \left(  \int_{B(x,r)} |\nabla a(y)|^n \dif y \right)^{1/n}
\end{equation}
and the fact that $\|f\|_{L^n_{\rm weak}} \sim \sup_E |E|^{\frac{1}{n}} \fint_E |f|$, see \cite{Giusti, Garafakos}.

The space $\VMO$ of functions of vanishing mean oscillation is the closure of $C_0^\infty(\R^n)$ in $\BMO$. We say $a \in \overline{\VMO}$ iff $a \in \BMO$ and 
\[
\lim_{R \searrow 0 } \norm{a}_{\BMO_R}=0.
\]
We note that $\overline{\VMO}$ space is strictly larger than the VMO space. 
The local BMO spaces are invariant under scaling $a_\lambda(x)=a(\lambda x)$, i.e.,
\[
\norm{a_\lambda}_{\BMO_R} = \norm{a}_{\BMO_{\lambda R}}.
\]

For the well-posedness result we introduce the space $ X_T$ for $T \in (0,+\infty)$, which consists of functions $u:\R^n \times [0,T] \to \R$ such that the following norm is finite:
\begin{equation}\label{eq:norm}
\norm{u}_{ X_T}:=\sum_{k=1}^2 N_{k,\infty, T}(u) + N_{k,c,T}(u)
\end{equation}
where
\[
N_{k,\infty, T}(u):= \sup_{0 < t \le T} t^{\frac{k}{4}} \norm{\nabla^k u(\cdot,t)}_{L^\infty(\R^n)} 
\]
and
\[
 N_{k,c,T}(u):= \sup_{\substack{x \in \R^n\\ 0<r \leq T^{\frac{1}{4}}}} 
\bra*{ \int_0^{r^4} \fint_{B(x,r)} \abs{\nabla^k u(y,t)}^{\frac{4}{k}} \dif y \dif t }^{\frac{k}{4}}.
\]

For $T=+\infty$, we define  $X$, or sometimes $X_\infty$, similarly by taking the supremum over all $t>0$ and $r>0$.
Space $X_T$ is invariant under the scaling \eqref{eq:scaling}. Note that we always identify functions that differ only by a constant in $X_T$ and $\BMO_R$ spaces.

Let $X_0$ (respectively $X_{0,T}$) be the closed subspace in $X$ (respectively $X_{0,T}$) of the functions $u$ such that
\[
\lim_{t \searrow 0} \norm{u}_{X_t}=0.
\]
Solutions of (\ref{eq:main}) may be characterised implicitly by Duhamel's formula
\[
u(t) = S(t)u_0(x) + \int_0^t S(t-s) \nabla \cdot F(\nabla u(s)) \dif s.
\]
More precisely, solutions to this integral equation in certain function spaces, here $X_T$, are called mild solutions of (\ref{eq:main}).  \\

Exploiting mapping properties and a fixed point argument we prove the following local and global well-posedness result for initial data with small local and global BMO norms, respectively.
\begin{theorem}\label{thm:existence}
Suppose $F$ satisfies \eqref{eq:lip}. There exists $\rho>0$ with the following property: 
\begin{enumerate}[(i)]
\item Given $T>0$ such that \[
\norm*{u_0}_{\BMO_{T^{1/4}}} < \rho, 
\]
then there exists a unique  small solution $u \in X_{T}$ of \eqref{eq:main} with $u(0)=u_0$, where the space $X_T$ is defined in \eqref{eq:norm}.
Moreover, if $u_0 \in \overline{\VMO}$, then there exists a time $T_0 >0$ such that there exists a unique solution   $u \in X_{0,T_0}$ of \eqref{eq:main} with $u(0)=u_0$.

\item If $\norm*{u_0}_{\BMO}< \rho$, then there exists a unique  small global solution  $u \in X$ of \eqref{eq:main} to initial data $u(0)=u_0$.
\end{enumerate}
\end{theorem}

The term \textit{small solution} refers to smallness in $X_T$. In this case, uniqueness simply follows from the contraction property of the fixed point equation.   
Uniqueness of a wider class of mild solutions in $X_{0,T}$ is discussed in Proposition \ref{prop:uniqueness}. \\

It is instructive to compare our criterion for existence with those of
Ishige et al. in \cite{ishige2022} Section 7.4 that translate into 
a smallness condition on either
\[
\sup_{\substack{x \in \R^n\\ 0<r \leq T^{\frac{1}{4}}}} 
    \fint_{B(x,r)} | u_0(y)| \dif y
\quad \text{or} \quad
 \sup_{\substack{x \in \R^n\\ 0<r \leq T^{\frac{1}{4}}}} 
  r  \fint_{B(x,r)} | \nabla u_0(y)| \dif y.
\]
Theorem \ref{thm:existence} improves on these criteria, because the space $\mathcal{M}_{1,n}$ is equivalent to $L^\infty$, which is strictly smaller than $\BMO$ (see \cite{stein}), and due to Poincar\'e's inequality \eqref{eq:poincare}.

Our second result concerns the stability of global solutions to \eqref{eq:main} with cubic nonlinearities. The analysis, essentially based on the decay of global solutions, reveals a fundamental difference between the coercive and non-coercive case. The coercive case features stability for possibly large global solutions, whose uniqueness follows from Proposition \ref{prop:uniqueness}. However, in non-coercive case, the possibility of non-trivial static solutions entails stability under certain smallness conditions. 

\begin{theorem}\label{thm:stability}
Let $n \ge 2$ and suppose one of the following conditions are fulfilled.
\begin{itemize}
\item[(a)] Let $F(\xi)= \abs{\xi}^2 \xi$ and $u_0 \in \VMO$ generating a (unique) global solution $u \in X_{0}$ of \eqref{eq:main}.
\item[(b)] Let $F(\xi)= - \abs{\xi}^2 \xi$ and $u_0 \in \VMO$ satisfy $\norm{u_0}_{\BMO} < \rho$, making $\rho$ smaller if necessary, and $u \in X_0$ be the solution of \eqref{eq:main} associated to $u_0$.
\end{itemize}
Then
\begin{enumerate}[(i)]
\item the global solution $u$ possesses following asymptotic property
\begin{equation}\label{eq:asym}
\lim_{t \to \infty} \norm{u (t+ \cdot)}_X = 0.
\end{equation}
\item There exists a $\delta=\delta(u)>0$ such that, for every $v_0\in \BMO$ with
\[
\norm{u_0 - v_0}_{\BMO} < \delta,
\]
\eqref{eq:main} admits a global solution $v \in X$ to initial data $v_0$, which satisfies
\[
\norm{u - v}_{X} < C \delta
\]
where $C>0$ is a constant depending on $u_0$.
\end{enumerate}
\end{theorem}

Applying Theorem \ref{thm:stability} to $u \equiv 0$, we recover the existence result in Theorem \ref{thm:existence} (ii) for the specific nonlinearity $F(\xi)= \pm \abs{\xi}^2 \xi$ and spatial dimensions $n \ge 2$.
In the coercive case, Theorem \ref{thm:stability} provides an improvement of the small data global existence result in Theorem \ref{thm:existence}.

Our discussion about global existence and regularity in the coercive case reveals the critical dimension $n=4$ where also the  governing energy functional \eqref{eq:energy} is scaling invariant. In the non-coercive case  $F(\xi)= -|\xi|^2 \xi$ this criticality is accompanied by the occurrence of non-trivial static solutions. Explicitly we have   
\[
u_0(x) = \pm 2 \ln \abs{x}.
\]
Although $u_0$ is a $\BMO$ function \cite{stein} (even in $\dot{W}^{1,4}_{\rm weak}(\R^4)$), it is neither a $\VMO$ function nor in $X_T$ for any $T>0$. Solutions that may be obtained by constrained minimization of $\|\nabla^2 u\|_{L^2}^2$ in the class $\|\nabla u\|_{L^4}=const.$ are necessarily large in $\VMO$ and $X_T$ and do not contradict the decay and stability result.

For the proof of Theorem \ref{thm:stability}, we adapt the method introduced by Auscher, Dubois, and Tchamitchian \cite{auscher2004stability} and Gallagher, Iftimie, and Planchon \cite{gallagher2003asymptotics} in the context of stability of Koch-Tataru solutions to the Navier-Stokes equation. 
An essential ingredient in their proof of stability is the cancellation property of the trilinear terms related to convection, which allows for an energy type estimate and, hence, the decay property of global solutions. Equation \eqref{eq:main} does not possess any similar cancellation property, and we need to distinguish between different cases.  In the coercive case $F(\xi) = \abs{\xi}^2 \xi$, we can use a similar energy argument as in \cite{auscher2004stability} to deduce the asymptotic behaviour of global solutions to any initial data in VMO. In the non-coercive case $F(\xi) = -\abs{\xi}^2 \xi$, instead, we prove uniform $L^p$-boundedness by a fixed-point argument and use Gagliardo-Nirenberg interpolation and the embedding of homogeneous Sobolev spaces into BMO space to derive the decay property of global solutions. Due to the fixed-point argument, this result is valid only for the solutions to small initial data.

The rest of the paper is organized as follows. In Section 2 we recall some basic properties of the biharmonic heat kernel. In Section 3 we prove Theorem \ref{thm:existence}, showing existence and uniqueness of solutions, and discuss their regularity properties. In section 4 we address the large time behaviour of the global solution and prove Theorem \ref{thm:stability}. Throughout the paper, we use $C_k$, $k=1,2,\cdots$, to denote constants with fixed values and use $C$ to denote constants that may vary from line to line.

\section{Preliminaries}
In this section, we review some fundamental properties on the biharmonic heat kernel and local BMO spaces.

The biharmonic heat kernel is the fundamental solution of the biharmonic heat equation
\[
(\partial_t + (-\Delta)^2) b(x,t) = 0 \quad \R^n \times (0,\infty),
\]
and it is given by
\[
b(x,t)=t^{-\frac{n}{4}} g\bra*{\frac{x}{t^{1/4}}}, \qquad b(\cdot,0)=\delta_0
\]
where $g$ is the Schwartz function given by
\[
g(x)=(2\pi)^{-\frac{n}{2}} \int_{\R^n} e^{i x \cdot \xi - \abs{\xi}^4} \dif \xi, \quad x \in \R^n.
\]

We collect some basic properties for the biharmonic heat kernel $b$, whose proof can be found in \cite{kochlamm}.
\begin{lemma}\label{la:basic}
\begin{enumerate}[(i)]
\item We have for every $t>0$ and $x \in \R^n$ that
\[
 \int_{\R^n} \nabla^k b(x-y,t) \dif y = 
\left\{\begin{array}{ll}
1, & k=0,\\ 0, & k \geq 1.
 \end{array} \right.
\]
\item For every $k, N\in \N_0$, there is a constant $c>0$ such that
\begin{equation}\label{eq:kernel_abs1}
 |\nabla^k b(x,t)| \leq C t^{-\frac{1}{4}(n+k-N)}(t^{\frac{1}{4}} + |x|)^{-N}
\end{equation}
for all $t>0$ and $x \in \R^n$, in particular, 
\begin{equation}\label{eq:kernel_abs2}
 |\nabla^k b(x,t)| \leq C(t^{\frac{1}{4}} + |x|)^{-n-k},
\end{equation}
and
\begin{equation}\label{eq:kernel_L1}
 \| \nabla ^k b(\cdot,t)\|_{L^1(\R^n)} \leq C t^{-\frac{k}{4}}.
\end{equation}
\end{enumerate}
\end{lemma}

Using \eqref{eq:kernel_L1} and the Young's inequality we obtain the following estimate in $L^p$-spaces for the semigroup $S(t)=e^{-t \lap^2}$:
\begin{lemma}\label{la:St}
Let $k \in \N_0$ and $1 \le p \le q \le +\infty$. Then there exists a constant $C=C(k,n)$ such that for any $t>0$ and $\varphi \in L^q(\R^n)$
\begin{equation}\label{eq:St}
\norm{\nabla^k S(t) \varphi}_{L^q(\R^n)} \le C t^{-\frac{n}{4}\bra*{\frac{1}{p}-\frac{1}{q}}-\frac{k}{4}} \norm{\varphi}_{L^p(\R^n)} .
\end{equation}
\end{lemma}

For the purpose of this paper, we recall the Carleson measure characterization of BMO spaces (see Stein \cite{stein} Theorem IV.3):

Let $\Phi: \R^n \to \R$ be a Schwartz function with $\int_{\R^n} \Phi =0$. For $t>0$, let $\Phi_t(x)=t^{-n}\Phi(\frac{x}{t})$. Suppose $\Phi$ is non-degenerate ,
\[
\int_{\R^n} \frac{ \abs{a(x)} \dif x}{1+\abs{x}^{n+1}} < \infty,
\]
and $\dif \mu = |\Phi_t * a(x)|^2 \frac{\dif x \dif t}{t}$ satisfies
\[
\sup_{x \in \R^n,r>0} \int_0^r \fint_{B(x,r)} |\Phi_t * a(x)|^2 \frac{\dif x \dif t}{t} < \infty.
\]
Then $a$ is in BMO and
\begin{equation}\label{la:BMOcarleson}
\norm{a}_{\BMO}^2
\leq
C \sup_{x \in \R^n,r>0} \int_0^r \fint_{B(x,r)} |\Phi_t * a(x)|^2 \frac{\dif x \dif t}{t} .
\end{equation}

Following the arguments in \cite{stein}, we have a similar Carleson measure characterization for local BMO spaces:
\begin{lemma}\label{la:BMO}
Let $\Phi$ be a Schwartz function with $\int_{\R^n} \Phi =0$. For $t>0$, let $\Phi_t(x)=t^{-n}\Phi(\frac{x}{t})$. 
If 
\[
 \sup_{\substack{x \in \R^n\\ 0<r \leq 2R}} \fint_{B(x,r)} \abs{a(y) - a_{x,r}} \dif y < \infty
\]
for some $R>0$, then
\begin{equation}\label{eq:BMOcarleson}
\sup_{\substack{x \in \R^n \\ 0 < r \le R}}  \int_0^r \fint_{B(x,r)} |\Phi_t * a(x)|^2 \frac{\dif x \dif t}{t} 
\leq C \bra*{\sup_{\substack{x \in \R^n\\ 0<r \leq 2R}} \fint_{B(x,r)} \abs{a(y) - a_{x,r}} \dif y}^2
\end{equation}
for some $C=C(n)>0$. In particular, we have
\[
 \norm{a}_{\BMO_{R}} \le C \sup_{\substack{x \in \R^n\\ 0<r \leq 2R}} \bra*{ \fint_{B(x,r)} \abs{a(y) - a_{x,r}} \dif y}.
\]
\end{lemma}
\begin{remark}\label{rem:bmo}
\begin{itemize}
\item[(1)] Actually we have
\[
\norm{a}_{\BMO_{R}} \le C_\lambda \sup_{\substack{x \in \R^n\\ 0<r \leq \lambda R}} \bra*{\fint_{B(x,r)} \abs{a(y) - a_{x,r}} \dif y} 
\]
for any fixed $\lambda>1$. For the simplicity, we choose $\lambda=2$.
\item[(2)] The inequality \eqref{eq:BMOcarleson} holds for any local BMO spaces on the right hand side with a radius larger than $R$.
Since \eqref{la:BMOcarleson} and \eqref{eq:BMOcarleson}, the space $\BMO_R$ for $R = \infty$ is identical to the standard BMO space. 
\end{itemize}
\end{remark}

\section{Well-posedness in local BMO spaces}

\subsection{Basic estimates for the biharmonic heat equation}

We first examine mapping properties of the biharmonic heat semigroup $S=S(t)$ in local $\BMO$ spaces.

\begin{lemma}\label{la:v0}
For $0 < T \leq +\infty$, if $a \in \BMO_{T^{1/4}}$, then there exists a constant $C=C(n)$ so that for $S(\cdot)a$ we have
\begin{equation}\label{eq:v0derivative}
N_{k, \infty, T }(S(\cdot)a) \le C\norm{a}_{\BMO_{T^{1/4}}} 
\end{equation}
and
\begin{equation}\label{eq:v0integral}
 N_{k, c,T}(S(\cdot)a) \le C \|a\|_{\BMO_{T^{1/4}}} \quad \text{for } k=1,2.
\end{equation}
\end{lemma}
\begin{proof}
Since both inequalities are invariant under translation and scaling (\ref{eq:scaling}), it suffices to show
\begin{equation}\label{eq:v01}
\abs{\nabla^k S(1)a(0)}
\leq C \|a\|_{\BMO_1} 
\end{equation}
and
\begin{equation}\label{eq:v02}
\bra*{ \int_0^1 \int_{B(0,1)} |\nabla^k S(s)a(y)|^{\frac{4}{k}} \dif y \dif s}^{\frac{k}{4}}
 \leq C \|a\|_{\BMO_1} \quad \text{for} \quad k =1,2.
\end{equation}
We write
\[
 S(1)a = S(1-s) S(s) a
\quad\text{and}\quad
 S(1)a = 2\int_0^{\frac{1}{2}} S(1-s) S(s) a \dif s.
\]
Since $S$ and $\nabla$ commute, it follows by the Cauchy-Schwarz inequality that
\[
\begin{aligned}
|\nabla^k S(1) a(0)|^2
 &= | 2\int_0^{\frac{1}{2}} \int_{\R^n} \nabla^{k-1} b(y,1-s) \nabla S(s)a(y) \dif y \dif s|^2\\
& \leq C \sum_{x \in \Z^n} | \int_0^{\frac{1}{2}} \int_{B(x,1)} \nabla^{k-1} b(y,1-s) \nabla S(s)a(y) \dif y \dif s |^2 \\
 &\leq C\sum_{x \in \Z^n} \Big( \int_0^{\frac{1}{2}} \int_{B(x,1)} s^{\frac{1}{2}}|\nabla^{k-1} b(y,1-s)|^2 \dif y \dif s\Big)\Big( \int_0^1\int_{B(x,1)} s^{-\frac{1}{2}}|\nabla S(s)a(y)|^2 \dif y \dif s \Big)\\
 &\leq C \|a\|_{\BMO_1}^2 \sum_{x \in \Z^n} \Big( \int_0^{\frac{1}{2}} \int_{B(x,1)} s^{\frac{1}{2}}|\nabla^{k-1}b(y,1-s)|^2 \dif y \dif s\Big)
\end{aligned}
\]
where the union of $B(x,1)$ with $x \in \Z^n$ covers $\R^n$. 

It follows by (\ref{eq:kernel_abs2}) that 
\[
\begin{aligned}
& \sum_{x \in \Z^n}\Big(\int_0^{\frac{1}{2}} \int_{B(x,1)} s^{\frac{1}{2}}|\nabla^{k-1}b(y,1-s)|^2 \dif y \dif s \Big)
 \leq \sum_{x \in \Z^n} \Big( \int_0^{\frac{1}{2}} \int_{B(x,1)} \frac{ \dif y \dif s }{((1-s)^{\frac{1}{4}} + |y|)^{2(n+k-1)}} \Big)\\
 = & \Big( \sum_{\substack{x \in \Z^n \\ |x| \geq 2 }} + \sum_{\substack{x \in \Z^n \\ |x| < 2}} \Big)\int_0^{\frac{1}{2}} \int_{B(x,1)} \frac{ 1 }{((1-s)^{\frac{1}{4}} + |y|)^{2(n+k-1)}} \dif y \dif s .
\end{aligned}
\]
For the second sum we have
\[
 \sum_{\substack{x \in \Z^n \\ |x| < 2}} \int_0^{\frac{1}{2}} \int_{B(x,1)} \frac{1}{((1-s)^{\frac{1}{4}} + |y|)^{2(n+k-1)}} \dif y \dif s 
 \leq \sum_{\substack{x \in \Z^n \\ |x| < 2}} \frac{1}{2} |B(x,1)| 2^{\frac{n+k-1}{2}} 
 \leq C.
\]
 On the other hand, since every $y \in \R^n$ can be covered by at most $2^n$ different balls $B(x,1)$ with $x \in \Z^n$ and every point in $B(0,1)$ can not be covered by any ball $B(x,1)$, $x \in \Z^n$ with $\abs{x} \geq 2$, we obtain for the first sum that
 \[
\begin{aligned}
& \sum_{\substack{x \in \Z^n \\ |x| \geq 2}} \int_0^{\frac{1}{2}} \int_{B(x,1)} \frac{1}{((1-s)^{\frac{1}{4}} + |y|)^{2(n+k-1)}} \dif y \dif s 
\leq \sum_{\substack{x \in \Z^n \\ |x| \geq 2}} \int_0^{\frac{1}{2}} \int_{B(x,1)} \frac{1}{|y|^{2(n+k-1)}} \dif y \dif s \\
\leq & 2^n \int_0^{\frac{1}{2}} \int_{\R^n \setminus B(0,1)} \frac{1}{|y|^{2(n+k-1)}} \dif y
\leq \frac{1}{2} 2^n n|B(0,1)| \int_1^{\infty} \frac{r^{n-1}}{r^{2(n+k-1)}} \dif y 
\leq C .
\end{aligned}
\]
Therefore we prove \eqref{eq:v01} and hence \eqref{eq:v0derivative}. This implies
\[
\begin{aligned}
\int_0^1 \int_{B(0,1)} |\nabla S(s)a(y)|^4 \dif y \dif s
\leq & C \|a\|_{\BMO_1}^2 \int_0^1 \int_{B(0,1)} s^{-\frac{1}{2}}|\nabla S(s)a(y)|^2 \dif y \dif s \\
\leq & C \|a\|_{\BMO_1}^4
\end{aligned}
\]
and
\[
\int_0^1 \int_{B(0,1)} |\nabla^2 S(s)a(y)|^2 \dif y \dif s
\leq C \|a\|_{\BMO_1}^2,
\]
hence \eqref{eq:v02} and \eqref{eq:v0integral} are proved.
\end{proof}

Lemma \ref{la:v0} implies the following estimate in $ X_T$:
\begin{lemma}\label{la:Sta}
For $0 <T \leq +\infty$, if $a \in \BMO_{T^{1/4}}$, then there exists a constant $C_1=C_1(n)$ so that
\[
\norm{S(\cdot)a}_{ X_T} \leq C_1 \|a\|_{\BMO_{T^{1/4}}}.
\]
Moreover, for every $\rho>0$ and $a\in \overline{\VMO}(\R^n)$ there exists $T_0>0$ such that
\begin{equation}\label{eq:v0vmo}
\norm{S(\cdot)a}_{X_{T_0}} < \rho.
\end{equation}
\end{lemma}

\subsection{Existence and uniqueness}
Based on the continuity property of the semigroup $S=S(t)$ we shall now, by means of a contraction argument, find a unique small solution $u$ to the integral equation
\[
u(x,t) = S(t)u_0(x) + G(u)(x,t)
\]
where
\[
G(u)(x,t)=\int_0^t \int_{\R^n} b(x-y,t-s) \nabla \cdot F(\nabla u(y,s)) \dif y \dif s,
\]
called mild solution of \eqref{eq:main}.

To this end, we first prove some mapping properties of $G$.
\begin{lemma}\label{la:G}
For $0 < T \leq +\infty$ and any $K_T>0$, if $u,v \in X_T$ with
\[
\|u\|_{ X_T}, \|v\|_{ X_T} \leq K_T,
\]
then $G(u), G(v) \in X_T$ and there exists a positive constant $C_2$, that
\begin{equation}\label{eq:G2}
\|G(u)-G(v)\|_{ X_T} \leq C_2 K_T^{2} \|u-v\|_{ X_T}.
\end{equation}
and
\begin{equation}\label{eq:G1}
\|G(u)\|_{ X_T} \leq C_2 \|u\|_{ X_T}^3.
\end{equation}
\end{lemma}

\begin{proof}
First we show that
\begin{equation}\label{eq:G01}
t^k |\nabla^k \big( G(u)(x,t)- G(v)(x,t)\big)| \leq C K_T^{2} \|u-v\|_{ X_T} \quad \text{for }k=0,1,2
\end{equation}
and 
\begin{equation}\label{eq:G02}
t^{-\frac{n}{k}} \| \nabla^k \bra*{ G(u)-G(v) } \|_{L^{\frac{4}{k}}(B(x,t^{\frac{1}{4}})\times[0,t])} \leq C K_T^{2} \|u-v\|_{ X_T} \quad \text{for }k=1,2
\end{equation}
and any $x \in \R^n$, $0<t \leq T$. Since both estimates are invariant under translation and scaling \eqref{eq:scaling}, we can assume $T \ge 1$, and it suffices to show that (\ref{eq:G01}) and (\ref{eq:G02}) hold for $x=0$ and $t=1$. Let $w=u-v$ and for $k=0,1,2$ we observe
\[
\begin{aligned}
& \nabla^k \Big( G(u)(0,1) - G(v)(0,1) \Big) 
= \int_0^1 \int_{\R^n} \nabla^k b(y,1-s) \nabla^2 w(y,s) F'(\nabla u(y,s)) \dif y \dif s \\
& \quad + \int_0^1 \int_{\R^n} \nabla^k b(y,1-s) \nabla^2 v(y,s) \Big( F'(\nabla u(y,s))- F'(\nabla v(y,s)) \Big) \dif y \dif s.
\end{aligned}
\]
We split the first integral into three parts
\[
\begin{aligned}
& \quad \int_0^1 \int_{\R^n} \Big| \nabla^k b(y,1-s) \nabla^2 w(y,s) F'(\nabla u(y,s))\Big| \dif y \dif s\\
\leq & \quad C\Big( \int_{\frac{1}{2}}^1 \int_{\R^n} + \int_0^{\frac{1}{2}} \int_{B_1} + \int_0^{\frac{1}{2}} \int_{\R^n \setminus B_1} \Big) 
|\nabla^k b(y,1-s)| \: |\nabla^2 w(y,s)| \: |\nabla u(y,s)|^2 \dif y \dif s\\
=:& \quad I_{11} + I_{12} + I_{13}.
\end{aligned}
\]
It follows by (\ref{eq:kernel_L1}) that
\[
\begin{aligned}
I_{11} \leq &\quad C \int_{\frac{1}{2}}^1 \|\nabla^2 w(s)\|_{L^{\infty}} \|\nabla u(s)\|_{L^{\infty}}^2 \int_{\R^n} |\nabla^k b(y,1-s)| \dif y \dif s\\
 \leq &\quad C \|w\|_{ X_T} \|u\|_{ X_T}^2 \int_{\frac{1}{2}}^1 s^{-1} \int_{\R^n} |\nabla^k b(y,1-s)| \dif y \dif s\\
\leq &\quad C \|w\|_{ X_T} \|u\|_{ X_T}^2 \int_{\frac{1}{2}}^1 s^{-1} (1-s)^{-\frac{k}{4}} \dif s
\leq C K_T^2 \|w\|_{ X_T},
\end{aligned}
\]
by (\ref{eq:kernel_abs2}) and the H\"older's inequality that
\[
\begin{aligned}
I_{12} \leq & \quad C \int_0^{\frac{1}{2}} \int_{B_1} \frac{1}{((1-s)^{\frac{1}{4}}+|y|)^{n+k}} |\nabla^2 w(y,s)| \: |\nabla u(y,s)|^2 \dif y \dif s\\
\leq & \quad C \Big( \int_0^{\frac{1}{2}} \int_{B_1} |\nabla^2 w(y,s)|^2 \dif y \dif s \Big)^{\frac{1}{2}}
\Big( \int_0^{\frac{1}{2}} \int_{B_1} |\nabla u(y,s)|^4 \dif y \dif s \Big)^{\frac{1}{2}}\\
\leq & \quad C \|w\|_{ X_T} \|u\|_{ X_T}^2
\end{aligned}
\]
and by (\ref{eq:kernel_abs1}) for $N=n+1$ and the H\"older's inequality that
\[
\begin{aligned}
I_{13} \leq &\quad C \int_0^{\frac{1}{2}} \int_{\R^n \setminus B_1} 
(1-s)^{-\frac{1}{4}(k-1)}((1-s)^{\frac{1}{4}}+|y|)^{-(n+1)} |\nabla^2 w(y,s)| \:|\nabla u(y,s)|^2 \dif y \dif s\\
\leq &\quad C \int_0^{\frac{1}{2}} \int_{\R^n \setminus B_1} |y|^{-(n+1)} |\nabla^2 w(y,s)| \:|\nabla u(y,s)|^2 \dif y \dif s\\
\leq &\quad C \sum_{j=1}^{\infty} \int_0^{\frac{1}{2}} \int_{B_{j+1} \setminus B_j} |y|^{-(n+1)} |\nabla^2 w(y,s)| \:|\nabla u(y,s)|^2 \dif y \dif s\\
\leq &\quad C \sum_{j=1}^{\infty} j^{n-1} j^{-(n+1)} 
\Big(\sup_{x \in \R^n}\int_0^{\frac{1}{2}} \int_{B(x,1)} |\nabla^2 w(y,s)| \:|\nabla u(y,s)|^2 \dif y \dif s \Big)\\
\leq &\quad C \Big( \sum_{j=1}^{\infty} j^{-2} \Big) \sup_{x \in \R^n} \Big( \int_0^{\frac{1}{2}} \int_{B(x,1)} |\nabla^2 w(y,s)|^2 \dif y \dif s \Big)^{\frac{1}{2}}
\Big( \int_0^{\frac{1}{2}} \int_{B(x,1)} |\nabla u(y,s)|^4 \dif y \dif s \Big)^{\frac{1}{2}}\\
\leq &\quad C \|w\|_{ X_T} \|u\|_{ X_T}^2.
\end{aligned}
\]
For the second integral, we have by (\ref{eq:lip})
\[
\begin{aligned}
& \quad \int_0^1 \int_{\R^n} \Big| \nabla^k b(y,1-s) \nabla^2 v(y,s) \Big( F'(\nabla u(y,s))- F'(\nabla v(y,s)) \Big) \Big|\dif y \dif s \\
&\leq \quad C \int_0^1 \int_{\R^n} |\nabla^k b(y,1-s)| \: |\nabla^2 v(y,s)| \: \Big( |\nabla u(y,s)|+|\nabla v(y,s)|\Big) |\nabla w(y,s)|\dif y \dif s
\end{aligned}
\]
and split them also into three parts:
\[
\begin{aligned}
&\quad \int_0^1 \int_{\R^n} |\nabla^k b(y,1-s)| \: |\nabla^2 v(y,s)| \: |\nabla u(y,s)|\: |\nabla w(y,s)| \dif y \dif s\\
\leq & \Big( \int_{\frac{1}{2}}^1 \int_{\R^n} + \int_0^{\frac{1}{2}} \int_{B_1} + \int_0^{\frac{1}{2}} \int_{\R^n \setminus B_1} \Big) |\nabla^k b(y,1-s)| \: |\nabla^2 v(y,s)| \: |\nabla u(y,s)|\: |\nabla w(y,s)| \dif y \dif s\\
=& I_{21} + I_{22} + I_{23}.
\end{aligned}
\]
It follows by (\ref{eq:kernel_L1}) that
\[
\begin{aligned}
I_{21} \leq & \int_{\frac{1}{2}}^1 \|\nabla^2 v(s)\|_{L^{\infty}} \|\nabla u(s)\|_{L^{\infty}} \|\nabla w(s)\|_{L^{\infty}} \int_{\R^n} |\nabla^k b(y,1-s)| \dif y \dif s\\
\leq & C \|v\|_{ X_T} \|u\|_{ X_T} \|w\|_{ X_T} \int_{\frac{1}{2}}^1 s^{-1}(1-s)^{-\frac{k}{4}} \dif s 
\leq C K_T^2 \|w\|_{ X_T},
\end{aligned}
\]
by (\ref{eq:kernel_abs2}) and the H\"older's inequality that
\[
\begin{aligned}
I_{22} &\leq \Big( \int_0^{\frac{1}{2}} \int_{B_1} |\nabla^k b(y,1-s)|^2 |\nabla^2 v(y,s)|^2 \Big)^{\frac{1}{2}}\Big( \int_0^{\frac{1}{2}} \int_{B_1} |\nabla u(y,s)|^2 |\nabla w(y,s)|^2 \Big)^{\frac{1}{2}}\\
&\leq C \Big( \int_0^{\frac{1}{2}} \int_{B_1} \frac{ |\nabla^2 v(y,s)|^2 }{((1-s)^{\frac{1}{4}}+|y|)^{2(n+k)}}\Big)^{\frac{1}{2}}\Big( \int_0^{\frac{1}{2}} \int_{B_1} |\nabla u(y,s)|^4 \Big)^{\frac{1}{4}}
\Big( \int_0^{\frac{1}{2}} \int_{B_1} |\nabla w(y,s)|^4 \Big)^{\frac{1}{4}}\\
&\leq C K_T^2\|w\|_{ X_T},
\end{aligned}
\]
and by (\ref{eq:kernel_abs1}) for $N=n+1$ and the H\"older's inequality that
\[
\begin{aligned}
I_{23} 
&\leq C\int_0^{\frac{1}{2}} \int_{\R^n \setminus B_1} 
(1-s)^{-\frac{1}{4}(k-1)}((1-s)^{\frac{1}{4}}+|y|)^{-(n+1)} |\nabla^2 v(y,s)| \: |\nabla u(y,s)| \: |\nabla w(y,s)| \\
&\leq C \int_0^{\frac{1}{2}} \int_{\R^n \setminus B_1} |y|^{-(n+1)} 
|\nabla^2 v(y,s)| \: |\nabla u(y,s)| \: |\nabla w(y,s)| \dif y \dif s\\
&\leq C \int_0^{\frac{1}{2}} \sum_{j=1}^{\infty} \int_{B_{j+1} \setminus B_j} |y|^{-(n+1)} 
|\nabla^2 v(y,s)| \: |\nabla u(y,s)| \: |\nabla w(y,s)| \dif y \dif s\\
&\leq C \sum_{j=1}^{\infty} j^{n-1}j^{-(n+1)} \sup_{x \in \R^n} \Big( \int_0^1 \int_{B(x,1)} |\nabla^2 v(y,s)|^2 \Big)^{\frac{1}{2}}\Big( \int_0^1 \int_{B(x,1)} |\nabla u(y,s)|^4 \Big)^{\frac{1}{4}} \\
&\qquad\qquad\qquad\qquad\qquad \cdot \Big( \int_0^1 \int_{B(x,1)} |\nabla w(y,s)|^4 \Big)^{\frac{1}{4}}\\
&\leq C\Big( \sum_{j=1}^{\infty} j^{-2} \Big) \|v\|_{ X_T} \|u\|_{ X_T} \|w\|_{ X_T}
\leq C K_T^2\|w\|_{ X_T}.
\end{aligned}
\]
Similarly we can show that
\[
\begin{aligned}
\int_0^1 \int_{\R^n}|\nabla^k b(y,1-s)| \: |\nabla^2 v(y,s)| \: |\nabla v(y,s)|\: |\nabla w(y,s)|\dif y \dif s
\leq CK_T^2\|w\|_{ X_T}
\end{aligned}
\]
and therefore
\[
|\nabla^k \Big( G(u)(0,1) - G(v)(0,1) \Big)| \leq CK_T^2 \|w\|_{ X_T}.
\]
For the integral estimate (\ref{eq:G02}) we use the energy method employed by Wang \cite{wang2012}. Denote $\Phi=G(u)-G(v)$ and 
\[
\Phi(x,t)=\int_0^t \int_{\R^n} b(x-y,t-s)\phi(y,s) \dif y \dif s, 
\]
where $h: \R^n \times [0,T] \to \R$ and
\[
\phi(y,s)= \nabla^2 w(y,s)F'(\nabla u(y,s)) + \nabla^2 v(y,s) \big(F'(\nabla u(y,s)) - F'(\nabla v(y,s)) \big)
\]
Then $W$ solves
\begin{equation}\label{eq:heatPhi}
(\partial_t + \lap^2) \Phi = \phi, \quad \Phi(t=0)=0.
\end{equation}
Select a smooth cutoff function $\eta \in C_0^{\infty}(B_2)$ such that $\eta =1$ on $B_1$ . Multiplying (\ref{eq:heatPhi}) by $\eta^4 \Phi$ and integrating over $\R^n \times [0,1]$, we obtain
\begin{equation}\label{eq:Phi1}
\int_{\R^n \times \{1\}} |\Phi|^2 \eta^4 + 2 \int_{\R^n \times [0,1]} \nabla^2 \Phi \cdot \nabla^2 (\Phi \eta^4) 
= \int_{\R^n \times [0,1]} \phi \cdot \Phi \eta^4.
\end{equation}
Since
\[
\begin{aligned}
\nabla^2(\Phi \eta^4) =& \nabla (\nabla (\Phi \eta^2 \cdot \eta^2)) 
= \nabla^2 (\Phi \eta^2) \eta^2 + 2 \nabla(\Phi \eta^2) \nabla(\eta^2) + (\Phi \eta^2) \nabla^2 (\eta^2),\\
\nabla^2 (\Phi \eta^2) =& \nabla^2 \Phi\cdot \eta^2 + 2 \nabla \Phi \cdot \nabla(\eta^2) + \Phi \nabla^2 (\eta^2)
\end{aligned}
\]
we have
\[
\begin{aligned}
\nabla^2 \Phi \cdot \nabla^2(\Phi \eta^4)
=& \nabla^2 \Phi (\nabla^2 (\Phi \eta^2) \eta^2 + 2 \nabla(\Phi \eta^2) \nabla(\eta^2) + (\Phi \eta^2) \nabla^2 (\eta^2))\\
=& \nabla^2 (W \eta^2)(\nabla^2 (\Phi \eta^2) - 2 \nabla \Phi \cdot \nabla(\eta^2) - \Phi \nabla^2 (\eta^2))\\
& + \nabla^2 \Phi (2 \nabla(\Phi \eta^2) \nabla(\eta^2) + (\Phi \eta^2) \nabla^2 (\eta^2)).
\end{aligned}
\]
Therefore, integrating by parts and using H\"older's inequality, we obtain
\[
\begin{aligned}
& \int_{\R^n \times [0,1]} \nabla^2 \Phi \cdot \nabla^2 (\Phi \eta^4)\\
=& \int_{\R^n \times [0,1]} |\nabla^2 (\Phi \eta^2)|^2 
 + \nabla^2 \Phi \Big( (\Phi \eta^2) \cdot \nabla^2(\eta^2) + 2\nabla(\Phi\eta^2) \cdot \nabla(\eta^2) \Big)\\
& - \int_{\R^n \times [0,1]}\nabla^2 (\Phi\eta^2) \Big( \Phi \cdot \nabla^2(\eta^2) + 2\nabla \Phi \cdot \nabla(\eta^2) \Big)\\
=& \int_{\R^n \times [0,1]} |\nabla^2 (\Phi \eta^2)|^2
- \int_{\R^n \times [0,1]} \nabla \Phi \cdot \nabla\Big((\Phi \eta^2) \cdot \nabla^2(\eta^2) + 2\nabla(\Phi\eta^2) \cdot \nabla(\eta^2)\Big)
\\
& - \int_{\R^n \times [0,1]}\nabla^2 (\Phi\eta^2) \Big( \Phi \cdot \nabla^2(\eta^2) + 2\nabla \Phi \cdot \nabla(\eta^2) \Big)\\
 \geq & \int_{\R^n \times [0,1]} |\nabla^2 (\Phi \eta^2)|^2 \\
&- C \Big( \int_{B_2 \times [0,1]} |\Phi|^2 + |\nabla \Phi|^2 + |\nabla \Phi|\: |\Phi| + |\nabla^2 \Phi| \: |\Phi| + |\nabla^2 \Phi| \: |\nabla \Phi| \Big)\\
 \geq & \int_{\R^n \times [0,1]} |\nabla^2 (\Phi \eta^2)|^2 
- C \Big( \int_{B_2 \times [0,1]} |\Phi|^2 + |\nabla \Phi|^2 + |\nabla^2 \Phi| \: |\Phi| + |\nabla^2 \Phi| \: |\nabla \Phi|\Big)\\
\end{aligned}
\]
and
\[
\begin{aligned}
& \|\phi\|_{L^1(B_2 \times [0,1])} \\
 \leq &
\|\nabla^2 w \: F'(\nabla u)\|_{L^1(B_2 \times [0,1])} + \| \nabla^2 v \Big( F'(\nabla u) - F'(\nabla v) \Big) \|_{L^1(B_2 \times [0,1])}\\
\leq &C\Big( \|\nabla^2 w \: |\nabla u|^2\|_{L^1(B_2 \times [0,1])} + \| |\nabla^2 v| (|\nabla u|+|\nabla v|)|\nabla w| \|_{L^1(B_2 \times [0,1])} \Big)\\
 \leq & C\Big( \|\nabla^2 w \|_{ L^2(B_2 \times [0,1])} \| \nabla u\|^2_{L^4(B_2 \times [0,1])} \\
&+ \| \nabla^2 v\|_{ L^2(B_2 \times [0,1])} \|\nabla u\|_{L^4(B_2 \times [0,1])}\|\nabla w\|_{L^4(B_2 \times [0,1])}\\ 
&+ \| \nabla^2 v\|_{ L^2(B_2 \times [0,1])} \|\nabla v\|_{L^4(B_2 \times [0,1])}\|\nabla w\|_{L^4(B_2 \times [0,1])}\Big)\\
\leq & C(\|w\|_{ X_T}\|u\|_{ X_T}^2 + \|w\|_{ X_T} \|u\|_{ X_T}\|v\|_{ X_T} + \|w\|_{ X_T} \|v\|_{ X_T}^2 )\\
\leq & C K_T^2 \|w\|_{ X_T}.
\end{aligned}
\]
Substituting the above two estimates into (\ref{eq:Phi1}) and using (\ref{eq:G01}), we have
\[
\begin{aligned}
& \int_{B_1 \times [0,1]} |\nabla^2 \Phi|^2
 \leq \int_{\R^n \times [0,1]} |\nabla^2 ( \Phi \eta^2)|^2 \\
 \leq & C \int_{B_2 \times [0,1]} \Big( |\Phi|^2 + |\nabla \Phi|^2 + |\nabla^2 \Phi| \: |\Phi| + |\nabla^2 \Phi| \: |\nabla \Phi| \Big)
+ \int_{\R^n \times[0,1]} \nabla^2 \Phi \cdot \nabla^2(\Phi \eta^4)\\
 \leq & C \Big[ \sup_{0 <t \leq 1}\|\Phi(t)\|^2_{L^{\infty}}+\Big( \int_0^1 t^{-\frac{1}{2}} \dif t \Big) \Big( \sup_{0 <t \leq 1} t^{\frac{1}{2}}\|\nabla \Phi(t)\|^2_{L^{\infty}} \Big) \\
& + \Big( \int_0^1 t^{-\frac{1}{2}} \dif t \Big) \Big( \sup_{0 <t \leq 1} t^{\frac{1}{2}}\|\nabla^2 \Phi(t)\|_{L^{\infty}} \|\Phi(t)\|_{L^{\infty}}\Big)\\
& + \Big( \int_0^1 t^{-\frac{3}{4}} \dif t \Big) \Big( \sup_{0 <t \leq 1} t^{\frac{1}{2}}\|\nabla^2 \Phi(t)\|_{L^{\infty}}\Big)\Big( \sup_{0 <t \leq 1} t^{\frac{1}{4}}\|\nabla \Phi(t)\|_{L^{\infty}}\Big)\Big]\\
& + \int_{B_2 \times [0,1]} |\phi|\: |\Phi| \eta^2\\
 \leq & C \Big[ \sup_{0 < t \leq 1} \Big(\| \Phi(t)\|^2_{L^{\infty}} 
+ t^{\frac{1}{2}}\|\nabla \Phi(t)\|^2_{L^{\infty}}
+ t^{\frac{1}{2}}\|\nabla^2 \Phi(t)\|_{L^{\infty}} \|\Phi(t)\|_{L^{\infty}}\\
&+ t^{\frac{1}{2}}\|\nabla^2 \Phi(t)\|_{L^{\infty}} t^{\frac{1}{4}}\|\nabla \Phi(t)\|_{L^{\infty}} \Big)
+ \|\phi\|_{L^1(B_2 \times [0,1])} \|\Phi\|_{L^{\infty}(B_2 \times [0,1])} \Big]\\
 \leq & C (K_T^2 \|w\|_{ X_T})^2.
\end{aligned}
\]
For the $L^4$ norm of $\nabla W$ on $B(0,1) \times (0,1)$, it follows from the Gagliardo-Nirenberg inequality that
\[
\| \nabla (\eta^2 \Phi(t)) \|^4_{L^4} \leq C \| \eta^2 \Phi(t) \|^2_{L^{\infty}} \| \nabla^2 (\eta^2 \Phi(t)) \|^2_{ L^2}.
\]
Integrating with respect to $t \in [0,1]$, we get
\[
\begin{aligned}
\Big(\int_{B(0,1) \times [0,1]} |\nabla \Phi|^4\Big)^{\frac{1}{4}} 
\leq & C \sup_{0 \leq t \leq 1} \|\Phi(t)\|^{\frac{1}{2}}_{L^{\infty}} \|\nabla^2(\eta^2 \Phi)\|^{\frac{1}{2}}_{ L^2(\R^n \times [0,1])} \\
\leq & C K_T^2 \|w\|_{ X_T}.
\end{aligned}
\]
This finishes the proof of (\ref{eq:G2}) which implies (\ref{eq:G1}).
\end{proof}

Using a fixed-point iteration argument, we conclude existence and uniqueness of small local or global mild solution of \eqref{eq:main}.
\begin{proposition}\label{prop:existence}
Suppose $f$ satisfies (\ref{eq:lip}). There exists a $\rho>0$ with the following property:
\begin{enumerate}[(i)]
\item For all $T>0$, if $\|u_0\|_{\BMO_{T^{1/4}}} < \rho$, then there exists a unique small mild solution $u$ in $X_T$ of \eqref{eq:main} to initial data $u(0)=u_0$, which satisfies
\[
\|u\|_{ X_T} < 2C_1 \|u_0\|_{\BMO_{T^{1/4}}}.
\]
Moreover, if $u_0 \in \overline{\VMO}(\R^n)$ then there exists a terminal time $T_0 >0$ such that there exists a unique small mild solution $u \in X_{T_0}$ of \eqref{eq:main} with $u(0)=u_0$.
\item If $\|u_0\|_{\BMO(\R^n)} < \rho$, then there exists a unique small global mild solution $u \in X$ of \eqref{eq:main} with $u(0)=u_0$.
\end{enumerate}
\end{proposition}

\begin{proof}
We define
\begin{equation}\label{eq:FP}
u_1 = S(t)u_0, \quad u_{j+1} = u_1 + G(u_j)
\end{equation}
and choose $\rho>0$, such that $4C_1^2 C_2 \rho^2 < \frac{1}{2}$. For $T \in (0,+\infty]$, if $\|u_0\|_{\BMO_{T^{1/4}}} < \rho$, we have
\begin{equation}\label{eq:delta_u0}
\|u_1\|_{ X_T} < \frac{\delta}{2}
\end{equation}
for $\delta:=2C_1 \rho$.
Note that if $u_0 \in \overline{\VMO}$, according to (\ref{eq:v0vmo}) in Lemma \ref{la:Sta} there exists a $T_0>0$ such that \eqref{eq:delta_u0} holds for $T=T_0$.

We prove by induction that
\[
\|u_j\|_{ X_T} < 2C_1 \|u_0\|_{\BMO_{T^{1/4}}} \quad \text{ for all } j \in \N.
\]
This is true for $j=1$ since Lemma \ref{la:Sta}. For the induction step, it follows by (\ref{eq:G1}) and our choice of $\delta$ that
\[
\begin{aligned}
\|u_{j+1}\|_{ X_T} 
\leq C_1 \|u_0\|_{\BMO_{T^{1/4}}} + C_2 \|u_j\|_{ X_T}^3
\leq 2C_1 \|u_0\|_{\BMO_{T^{1/4}}}.
\end{aligned}
\]
Therefore we have, for all $j \in \N_0$,
\[
\|u_j\|_{ X_T} < 2C_1 \|u_0\|_{\BMO_{T^{1/4}}} < \delta.
\]
(\ref{eq:G2}) implies
\[
 \|G(u_{j+1}) - G(u_j)\|_{ X_T}
\leq C_2 \delta^2 \|u_{j+1} - u_j\|_{ X_T} 
 < \frac{1}{2} \|u_{j+1} - u_j\|_{ X_T}
\]
and hence
\[
\| u_{j+1}-u_j\|_{ X_T} = \| G(u_j)-G(u_{j-1})\|_{ X_T} 
\leq \frac{1}{2} \| u_j - u_{j-1}\|_{ X_T} \leq \cdots \leq 2^{-j} \| u_1 -u_0\|_{ X_T}.
\]
For $k>i>0$ we have
\[
\begin{aligned}
\|u_k-u_i\|_{ X_T} & \leq & \sum_{j=i}^{k-1} \| u_{j+1}-u_j \|_{ X_T} 
\leq \sum_{j=i}^{k-1} 2^{-j} \| u_1 -u_0\|_{ X_T}
\leq C \delta 2^{-i}.
\end{aligned}
\]
Therefore $(u_i)_{i \in \N}$ is a Cauchy sequence in $ X_T$ and converges to a mild solution $u \in X_T$ of (\ref{eq:main}) and
\[
\|u\|_{ X_T} < 2C_1 \|u_0\|_{\BMO_{T^{1/4}}}.
\]
For $a \in \overline{\VMO}(\R^n)$, we obtain by \eqref{eq:G1} that
\[
\lim_{t \searrow 0}  \|G(u)\|_{X_t}
\leq C \lim_{t \searrow 0}\|u\|_{ X_t}^3 
\leq C \lim_{t \searrow 0} \|u_0\|_{\BMO_{t^{1/4}}}^3 = 0
\]
and therefore  $\lim_{t \searrow 0}\|u\|_{ X_t}=0$.
\end{proof}

The above argument only guarantees the uniqueness of mild solutions that are sufficiently small in $X_T$. Now we study the uniqueness of general mild solutions in $X_{0,T}$.
A straightforward calculation shows the following result similar to Lemma 7 in \cite{auscher2004stability}.
\begin{lemma}\label{la:Nshift}
Let $0<t<t'<+\infty$ and $\tilde{u}=u(t+\cdot)$, defined on $(0,t'-t) \times \R^n$. Then
\[
N_{k,c,t'-t}(\tilde{u}) \le \bra*{\ln\frac{t'}{t}}^{\frac{k}{4}} N_{k,\infty}(u)
\quad \text{and}\quad
N_{k,\infty,t'-t}(\tilde{u}) \le \bra*{1-\frac{t}{t'}}^{\frac{k}{4}} N_{k,\infty}(u)
\]
for $k=1,2$.
\end{lemma}

We apply this lemma to prove uniqueness of mild solutions in $X_{0,T}$.
\begin{proposition}\label{prop:uniqueness}
Let $T>0$. Suppose $u,v \in X_{0,T}$ are mild solutions to \eqref{eq:main} for identical initial conditions $u_0 \in\BMO$, then $u$ and $v$ coincide in $X_{0,T}$.
\end{proposition}
\begin{proof}
Let $\delta>0$ be a fixed constant, whose value will be decided later. It follows from Lemma \ref{la:Nshift} that, there exist $J \in \N$ and $J+1$ overlapping intervals $I_j = (t_j, t_j')$ covering $(0,T)$, with $t_0=0$, $t_J' = T$ and $t_j < t_{j-1}'$ for any $1 \le j \le J$, such that
\[
\norm{\tilde{u}_j} _{X_{t_j'-t_j}}, \norm{\tilde{v}_j} _{X_{t_j'-t_j}} \le \delta \quad\text{for any } 0 \le j \le J
\]
for
\[
\tilde{u}_j(t)=u(t_j + t), \quad \tilde{v}_j(t)=v(t_j + t) \quad t \in (0,t_j'-t_j).
\]
 Since
\[
u(t)=S(t-t_j) u(t_j) + \int_{0}^{t-t_j} S(t-t_j-s) \nabla \cdot F(\nabla u(t_j+s)) \dif s \quad\text{for}\quad t \in  (t_j, t_j'),
\]
we have
\[
\tilde{u}_j(t)=S(t)u(t_j)+G(\tilde{u}_j)(t) \quad\text{for}\quad t \in  (0, t_j'-t_j).
\]
Set $\tilde{w}_0:=\tilde{u}_0-\tilde{v}_0$. Since $u(t_0)=v(t_0)=u_0$, it follows by \eqref{eq:G2} that
\[
\norm{\tilde{w}_0} _{X_{t_0'-t_0}}=\norm{G(\tilde{u}_0)-G(\tilde{v}_0)} _{X_{t_0'-t_0}} \le C_2 \delta^2 \norm{\tilde{w}_0} _{X_{t_0'-t_0}}.
\]
Taking small enough $\delta>0$ so that $C\delta^2 \le \frac{1}{2}$, we deduce 
\[
\norm{\tilde{w}_0} _{X_{t_0'-t_0}} \le \frac{1}{2} \norm{\tilde{w}_0} _{X_{t_0'-t_0}},
\]
so that $u$ and $v$ coincide on $\R\times I_0$, in particular $u(t_1)=v(t_1)$. Repeating this process inductively for $ j =1, \cdots, J$, the claim follows.
\end{proof}
 
\subsection{Higher regularity}

If in addition $F \in C^\infty(\R^n;\R^n)$, then the solution $u$ obtained in Theorem \ref{thm:existence} satisfies
\[
N_{k,\infty,T}(u) = \sup_{k \in \N} \, t^{\frac{k}{4}} \norm{\nabla^k u}_{L^\infty(\R^n)}<\infty \quad\text{for any}\quad k \ge 1.
\]
The proof follows from an upgrade of the fixed point argument above, cf. \cite{germain2007regularity}
for Navier-Stokes. In the following we outline the main idea and give a sketch of proof.

For $\ell \in \N$ and $\ell \ge 2$ we define the space $X_T^\ell$ of functions $u:\R^n \times [0,T] \to \R$ such that
\[
\norm{u}_{X_T^\ell}= \sum_{k=1}^\ell N_{k,\infty,T}(u) + N_{k,c,T}(u) < \infty
\]

It follows by similar arguments as in the proof Lemma \ref{la:v0} that
\[
N_{k,\infty,T}(S(\cdot)u_0), N_{k,c,T}(S(\cdot)u_0) \le C \norm{u_0}_{\BMO_{T^{1/4}}} \quad\text{for all}\quad k \le \ell.
\]
 Following the idea in \cite{germain2007regularity}, one can adapt the argument in the proof of Lemma \ref{la:G} to prove that
\[
N_{\ell+1,\infty,T}(G(u)) , N_{\ell+1,c,T}(G(u)) \le C \norm{u}_{X_T^\ell}^3 \quad\text{for all}\quad \ell \ge 2.
\]
By induction we prove that the sequence $(u_j)_{j \in \N}$ defined in \eqref{eq:FP} satisfies
\[
N_{\ell,\infty,T}(u_j) +N_{\ell,c,T}(u_j)  \le C_\ell
\]
where $C_\ell >0$ depending on $\ell$, but independent of $j$. This implies that $(u_j)_{j \in \N}$ is uniformly bounded in $X_T^\ell$, hence its limit $u$ is in $X_T^\ell$.

\section{Stability of global solutions}
\subsection{Main idea}
This section is devoted to the proof of Theorem \ref{thm:stability}, i.e., asymptotics and  stability of global solutions in the case of cubic nonlinearities $F(\xi)=\pm \abs{\xi}^2 \xi$.  Key is a linearization principle by Auscher, Dubois and Tchamitchian's in \cite{auscher2004stability} introduced in the context 
of global solutions to the Navier-Stokes equation in $\R^3$ to initial data in $\VMO^{-1}$. An essential feature of their method is the cancellation property of the trilinear term specific to Navier-Stokes. The trilinear term $\nabla \cdot F(\nabla u) $ in our equation does not possess any similar cancellation property. However, we will prove that, under certain conditions, global solutions of \eqref{eq:main} to initial data in $\VMO$ possess some asymptotic behaviour, which is the key to the stability of the solutions. Depending on the property of energy functional, we distinguish the following two cases:
\begin{itemize}
\item[(a)] For the coercive case $F(\xi)=\abs{\xi}^2 \xi$, we derive an energy-type inequality and use interpolation to deduce the asymptotics of global solutions of \eqref{eq:main} to initial data in $\VMO$ without a size restriction. 
\item[(b)] For the non-coercive case $F(\xi)=-\abs{\xi}^2 \xi$, though there is no similar energy-type inequality, we use a fixed-point argument to obtain the asymptotic result of global solutions, for initial data small enough in $\VMO$.
\end{itemize}

We define the trilinear operator
\[
\Psi(f,g,h)(t) = \int_0^t S(t-s) \psi(f(s),g(s),h(s)) \dif s
\]
where
\[
\psi(f,g,h):= \nabla \cdot  \bra*{ (\nabla f \cdot \nabla g ) \nabla h} =   (\nabla f \cdot \nabla g ) \lap h + \nabla f \cdot \nabla^2 g \cdot \nabla h + \nabla g \cdot \nabla^2 f \cdot \nabla h
\]
and we denote
\[
\nabla f \cdot \nabla^2 g \cdot \nabla h = \sum_{i=1}^n \sum_{k=1}^n \partial_k f \cdot \partial_i \partial_k g \cdot \partial_i h.
\]
Then mild solutions of \eqref{eq:main} for $F(\xi)=\pm \abs{\xi}^2 \xi$ are characterized by
\[
u(t)=S(t)u_0 \pm \Psi(u,u,u)(t).
\]
The strategy of proof is to show that for any initial data $v_0 \in \BMO$ near $u_0$, \eqref{eq:main} has a solution $v \in X$. This is equivalent to solving the following equation for $w=u-v$ to small initial data $w_0:= u_0 - v_0$ 
\begin{equation}\label{eq:w1}
w(t) = S(t)w_0 - \mathcal{L}_{u,u} w(t)+ \mathcal{B}_u(w,w)(t)- \Psi(w,w,w)(t)
\end{equation}
in terms of linear and bilinear operators 
\[
\mathcal{L}_{a,b} f := \pm (\Psi(f,a,b) + \Psi(a,b,f) + \Psi(b,f,a))
\]
and 
\[
\mathcal{B}_a(f,f) := \mathcal{L}_{a,f} f,
\]
respectively, acting on $f \in X$ for $a,b \in X$.
Making use of the asymptotic behaviour of $u$, we prove that the operator $I+ \mathcal{L}_{u,u}$ is invertible on $X$. Then the equation \eqref{eq:w1} is equivalent to
\[
w(t) = (I+ \mathcal{L}_{u,u})^{-1} S(t)w_0 + (I+ \mathcal{L}_{u,u})^{-1} \pra*{\mathcal{B}_u(w,w)(t)- \Psi(w,w,w)(t)},
\]
which is solvable in $X$ for $w_0$ small enough in $\BMO$ according to the following lemma.
\begin{lemma}\label{la:NL}
Let $X$ be a Banach space, $\mathcal{L}$ a linear operator, $\mathcal{B}$ a bilinear operator and $\Psi$ a trilinear operator, continuous on $X$ with bounded operator norm. Suppose $I+\mathcal{L}$ is invertible. Then for $z \in X$ such that
\[
\norm{(I+\mathcal{L})^{-1}z}_X < \frac{1}{4(\norm{\mathcal{B}}+\norm{\Psi}) \norm{(I+\mathcal{L})^{-1}}},
\]
there is a solution $w \in X$ to the equation
\begin{equation}\label{eq:w}
 w = z - \mathcal{L}w + \mathcal{B}(w,w) - \Psi(w,w,w).
\end{equation}
Furthermore this solution satisfies
\[
\norm{w}_X \le C_z \norm{z}_X,
\]
where $C_z>0$ is a constant depending on $z$.
\end{lemma}
\begin{proof}
First we assume $\norm{\mathcal{L}}<1$ and define a sequence
\[
w_1 = \tilde{z}, \quad \tilde{w}_{n+1} = \tilde{w}_1 + \mathcal{L} \tilde{w}_n+B( \tilde{w}_n, \tilde{w}_n)+\Psi( \tilde{w}_n, \tilde{w}_n, \tilde{w}_n).
\]
It is straightforward to prove that, if $\norm{\tilde{z}}_X \le \frac{(1-\norm{\mathcal{L}})^2}{4(\norm{B}+\norm{\Psi})}$, then the sequence $\norm{ \tilde{w}_n}_X$ is uniformly bounded and the telescopic series $\sum \norm{ \tilde{w}_{n+1}- \tilde{w}_n}_X$ converges. This gives the solution $ \tilde{w}=\sum( \tilde{w}_{n+1}- \tilde{w}_n)$ to \eqref{eq:w}.

In the general case, when $I+\mathcal{L}$ is invertible, equation \eqref{eq:w} is equivalent to
\[
w=(I+\mathcal{L})^{-1}z+(I+\mathcal{L})^{-1}B(w,w)-(I+\mathcal{L})^{-1}\Psi(w,w,w).
\]
Applying above result on $\mathcal{L}=0$ and $\tilde{z}=(I+\mathcal{L})^{-1}z$ yields solution of \eqref{eq:w}.
\end{proof}

\subsection{The linear operator $\mathcal{L}_{a,b}$}
In preparation for the proof of stability, we first prove boundedness of trilinear operator $\Psi$ in $X_T$. 
\bigskip
\begin{lemma}\label{la:Lcontinuity}
Let $T \in (0,+\infty]$.
\begin{enumerate}[(i)] 
\item $\Psi$ is well-defined and continuous on $X_T \times X_T \times X_T$ with
\[
 \norm{\Psi(f,g,h)}_{X_T} \le C \norm{f}_{X_T} \norm{g}_{X_T} \norm{h}_{X_T}.
\]
\item Let $a,b \in X_T$. Then $\mathcal{L}_{a,b}$ is continuous on $X_T$ with
\[
 \norm{\mathcal{L}_{a,b} f}_{X_T} \le C \norm{a}_{X_T} \norm{b}_{X_T} \norm{f}_{X_T}.
\]
\end{enumerate}
\end{lemma}
\begin{proof}
Let $f,g,h \in X_T$. Taking into account the beta integral
\begin{equation}\label{eq:beta}
\int_0^t (t-s)^{-\alpha} s^{-\beta} \dif s = c t^{1-\alpha-\beta} \quad \text{for }\alpha, \beta \in (0,1)
\end{equation}
and \eqref{eq:St}, we obtain for $k=1,2$ that
\[
\begin{aligned}
&\int_0^t \norm{\nabla^k S(t-s) \Psi(f,g,h)(t)}_{L^\infty(\R^n)} \\ 
\le &  \int_0^t (t-s)^{-\frac{k+1}{4}} \bigg(\norm{\nabla f(t)}_{L^\infty}\norm{\nabla g(s)}_{L^\infty}\norm{\nabla h(t)}_{L^\infty}\bigg) \dif s\\
 \le &  \int_0^t (t-s)^{-\frac{k+1}{4}} s^{-\frac{3}{4}} \dif s  N_{1,\infty,T}(f)  N_{1,\infty,T}(g) N_{1,\infty,T}(h) \\
 \le & C t^{-\frac{k}{4}} \norm{f}_{X_T}\norm{g}_{X_T}\norm{h}_{X_T},
\end{aligned}
\]
hence, for the estimate of $\norm{\nabla \Psi(f,g,h)}_{L^4(B(x,r) \times (0,r^4))}$, we can assume $x=0$ and $r=1$ due to the invariance under translation and scaling. We note that $\Psi(f,g,h)$ solves \eqref{eq:heatPhi} for $\phi=\psi(f,g,h)$.
It follows by the H\"older inequality that
\[
\begin{aligned}
\norm{\psi(f,g,h)}_{L^1(B_2 \times [0,1])} 
\le & C \bra{\norm{\nabla f}_{L^4(B_2 \times [0,1])} \norm{\nabla g}_{L^4(B_2 \times [0,1])} \norm{\lap h}_{L^2(B_2 \times [0,1])}\\
&\quad +  \norm{\nabla f}_{L^4(B_2 \times [0,1])} \norm{\nabla^2 g}_{L^2(B_2 \times [0,1])} \norm{\nabla h}_{L^4(B_2 \times [0,1])} \\
&\quad  +  \norm{\nabla^2 f}_{L^2(B_2 \times [0,1])} \norm{\nabla g}_{L^4(B_2 \times [0,1])} \norm{\nabla h}_{L^4(B_2 \times [0,1])}}\\
 \le & C \norm{f}_{X_T} \norm{g}_{X_T} \norm{h}_{X_T}.
\end{aligned}
\]
Using the same argument to prove \eqref{eq:G02} in Lemma \ref{la:G} with help of a smooth cutoff function of $B_1$, we obtain
\[
\norm{\nabla^2 \Psi}_{L^2(B_1 \times [0,1])} \le C \norm{f}_{X_T} \norm{g}_{X_T} \norm{h}_{X_T}
\]
and 
\[
\norm{\nabla \Psi}_{L^4(B_1 \times [0,1])}  \le C \norm{f}_{X_T} \norm{g}_{X_T}\norm{h}_{X_T}.
\]
Hence
\[
\norm{\Psi(f,g,h)}_{X_T} \le C \norm{f}_{X_T} \norm{g}_{X_T} \norm{h}_{X_T}
\]
where $C$ is a universal constant independent of $f,g,h$ and $T$. Thus (i) follows and from this directly (ii).
\end{proof}

The operator $\Psi$ preserves BMO norm in the following sense:
\begin{lemma}\label{la:Lpersist}
Let $a,b,f \in X_T$ for some $T \in (0,+\infty]$. Then $\Psi(a,b,f) \in \BMO$ for any $t \in (0,T)$ and
\[
\norm{\Psi(a,b,f)(t)}_{\BMO} \le C\norm{a}_{X_T} \norm{b}_{X_T} \norm{f}_{X_T}
\]
for some universal constant $C>0$ independent of $a,b,f$ and $t$.
\end{lemma}
\begin{proof}
Due to the embedding $L^{\infty} \subset \BMO$, it is enough to show
\begin{equation}\label{eq:Linfty}
\norm{\Psi(a,b,f)(t)}_{L^{\infty}} \le C \norm{a}_{X_T} \norm{b}_{X_T} \norm{f}_{X_T}
\end{equation}
for some uniform constant $C>0$. It follows by \eqref{eq:St}  and \eqref{eq:beta} that
\[
\begin{aligned}
& \norm{\int_0^t S(t-s) \psi(f(s), a(s), b(s)) \dif s}_{L^{\infty}} \le \int_0^t \norm{ \nabla S(t-s) \bra*{ (\nabla f \cdot \nabla a ) \nabla b} }_{L^{\infty}}  \dif s\\
\le &\,  C\int_0^t  (t-s)^{-\frac{1}{4}} s^{-\frac{3}{4}} \bigg[ \bra*{s^{\frac{1}{4}}\norm{\nabla a(s)}_{L^\infty}}\bra*{s^{\frac{1}{4}}\norm{\nabla b(s)}_{L^\infty}} \bra*{s^{\frac{1}{4}}\norm{\nabla f(s)}_{L^p}} \bigg] \dif s \\
\le &\,C  \norm{a}_{X_T} \norm{b}_{X_T} \norm{f}_{X_T}.
\end{aligned}
\]
Therefore we have proved \eqref{eq:Linfty} and thus the claim.
\end{proof}

We need an auxiliary space $L_T^p$ for $T \in (0,+\infty]$ and $p \in [1, +\infty]$ of functions $f$ such that
\[
\norm{f}_{L_T^p} := \sup_{0 <t \le T} \bra*{\norm{f(t)}_{ L^p(\R^n)} + t^{\frac{1}{4}} \norm{\nabla f(t)}_{L^p(\R^n)}+ t^{\frac{1}{2}} \norm{\nabla^2 f(t)}_{L^p(\R^n)}} < +\infty.
\]

\bigskip
\begin{lemma}\label{la:LcontinuityL}
Let $a,b \in X_T$ for some $T \in (0,+\infty]$. Then
\begin{enumerate}[(i)]
\item $\Psi(\cdot, a,b) , \Psi(b, \cdot, a), \Psi(a,b,\cdot) : L_T^p \to L_T^p$ are well-defined and continuous with
\[
 \norm{\Psi(f, a,b)}_{L_T^p}, \norm{\Psi(b,f, a)}_{L_T^p}, \norm{\Psi(b,a, f)}_{L_T^p}  \le C_3 \norm{a}_{X_T} \norm{b}_{X_T} \norm{f}_{L_T^p}
\]
for a uniform constant $C_3>0$.
\item $\mathcal{L}_{a,b}$ is continuous on $L_T^p$ with
\[
 \norm{\mathcal{L}_{a,b} f}_{L_T^p} \le C_4 \norm{a}_{X_T}   \norm{b}_{X_T}  \norm{f}_{L_T^p}
\]
for a uniform constant $C_4>0$.
\end{enumerate}
\end{lemma}

\begin{proof}
Let $f \in L_T^p$. It follows by  \eqref{eq:St} and \eqref{eq:beta} that, for any $t \in (0,T]$ and $k=0,1,2$ 
\[
\begin{aligned}
& \norm{\nabla^k\Psi(f,a,b)(t)}_{L^p} \le \int_0^t \norm{\nabla^{k+1} S(t-s) (\nabla f(s) \cdot \nabla a(s)) \nabla b(s))}_{L^p} \dif s \\
& \le \int_0^t C(t-s)^{-\frac{k+1}{4}} \norm{ (\nabla f(s) \cdot \nabla a(s)) \nabla b(s)}_{L^p} \dif s\\
& \le C \int_0^t (t-s)^{-\frac{k+1}{4}} s^{-\frac{3}{4}} \bigg[ \bra*{s^{\frac{1}{4}}\norm{\nabla a(s)}_{L^\infty}}\bra*{s^{\frac{1}{4}}\norm{\nabla b(s)}_{L^\infty}} \bra*{s^{\frac{1}{4}}\norm{\nabla f(s)}_{L^p}} \bigg] \dif s \\
& \le C  t^{-\frac{k}{4}}N_{1,\infty,T}(a) N_{1,\infty,T}(b)  \norm{f}_{L_T^p},
\end{aligned}
\]
thus 
\[
 \norm{\Psi(f,a,b)}_{L_T^p} \le C  N_{1,\infty,T}(a)  N_{1,\infty,T}(b)\norm{f}_{L_T^p}.
\]
The corresponding estimates for $\Psi(a,b,f)$ and $\Psi(b,f,a)$ follows similarly and implies the continuity of $\mathcal{L}_{a,b}$ on $L_T^p$.
\end{proof}

Solving \eqref{eq:w1} relies on the spectral properties of $\mathcal{L}_{a,b}$.
\begin{lemma}\label{la:spectrum}
Let $a,b \in X_{0,T}$ for some $T \in (0,+\infty]$, with the property that
\begin{equation}\label{eq:asym_ab_inf}
\lim_{t \to \infty} \norm{a (t+ \cdot)}_X =0 \quad\text{or}\quad \lim_{t \to \infty} \norm{b (t+ \cdot)}_X = 0
\end{equation}
when $T=+\infty$, then the operator $I+\mathcal{L}_{a,b}$ is invertible on $X_T$.
\end{lemma}

\begin{proof}
We may only consider the case $T=+\infty$. For $T<+\infty$, we can extend $a$ and $b$ by $0$ on $[T,+\infty)$. We want to show the equation
\begin{equation}\label{eq:spec}
f + \mathcal{L}_{a,b}(f) = g
\end{equation}
has a unique solution in $X$ for any $g \in X$. The strategy is to construct the global solution to \eqref{eq:spec} from finite many local solutions. Let $\delta>0$ be a fixed constant, whose value will be decided later. We deduce from Lemma \ref{la:Nshift}, the assumption \eqref{eq:asym_ab_inf} and the fact that $a,b \in X_{0,T}$, the existence of an integer $J \in \N$ and $J+1$ overlapping intervals $I_j = (t_j, t_j')$ covering $(0,+\infty)$, with $t_0=0$, $t_J' = +\infty$ and $t_j < t_{j-1}'$ for any $1 \le j \le J$, such that
\[
\norm{a_j} _{X_{t_j'-t_j}}, \norm{b_j} _{X_{t_j'-t_j}} \le \delta \quad\text{for any } 0 \le j \le J
\]
for
\[
a_j(t)=a(t_j + t), \quad b_j(t)=b(t_j + t) \quad t \in (0,t_j'-t_j).
\]
According to Lemma \ref{la:Lcontinuity},  we may choose $\delta$ small enough such that
\begin{equation}\label{eq:LnormX}
\norm{\mathcal{L}_{a_j,b_j}}_{X_{t_j'-t_j}} \le \frac{1}{2} \quad \text{for } j = 0,1, \cdots, J,
\end{equation}
thus $I+\mathcal{L}_{a_j,b_j}$ is invertible on $X_{t_j'-t_j}$.
Now we construct the global solution to \eqref{eq:spec} in $X$ by starting with \eqref{eq:spec} restricted on $I_0=(0,t_0')$. According to \eqref{eq:LnormX}, there exists a unique $f_0 \in X_{t_0'}$ such that
\[
f_0(t) + \mathcal{L}_{a_0,b_0} f_0(t) = g(t) \quad \forall t \in I_0.
\]
Again using \eqref{eq:LnormX}, there exists a unique $\tilde{f_1} \in X_{t_1'-t_1}$ such that
\[
\tilde{f}_1(\tau) + L_{a_1,b_1} \tilde{f}_1(\tau) = g(t_1 + \tau) + S(\tau) L_{a_0,b_0} f_0(t_1) \quad \forall \tau \in (0,t_1'-t_1).
\]
Note that $L_{a_0,b_0} f_0(t_1) \in \BMO$ by Lemma \ref{la:Lpersist}, hence the right hand side of above equation is in $X_{t_1'-t_1}$ according to Lemma \ref{la:v0}. Now we can define
\[
f_1(t) := \left\{ \begin{array}{ll}
f_0(t) , & \text{for } t \in I_0\\
f_1(t)=\tilde{f}_1(t-t_1) , & \text{for } t \in I_1
\end{array}\right.
\]
which is a unique solution to \eqref{eq:spec} on $I_0 \cup I_1$ and $f_1 \in X_{t_1'}$. Iterating this process $J-1$ times, we obtain the unique solution $f=f_{J} \in X$ to \eqref{eq:spec} on $(0,+\infty)$.
\end{proof}

\subsection{Asymptotics of $u$}
In this section we consider the asymptotic behaviour of the global solution. This implies that \eqref{eq:asym_ab_inf} is fulfilled for $a=b=u$.

Let $\eps>0$ be fixed and $u_0 \in \VMO$ generating a 
unique global solution $u \in X_0$ of \eqref{eq:main} to initial data $u_0$, cf. Proposition \ref{prop:uniqueness}. According to the definition of VMO space, there exists a decomposition:
\[
u_0 = f_0 + g_0, \quad f_0 \in C_0^{\infty}(\R^n) \quad\text{and}\quad g_0 \in \VMO,
\]
where $\norm{g_0}_{\BMO}$ is small enough, so that there exists a unique global solution  $g \in X_0$  of \eqref{eq:main} to initial data $g_0$, cf. Theorem \ref{thm:existence}. 
By Lemma \ref{la:g} below, $g(t)$ stays in $\BMO$ and $\norm{g(t)}_{\BMO} \le C\eps$. For $f=u-g$, if there exists a time $T'>0$ at which $\norm{f(T')}_{\BMO} \le \eps$, then $\norm{u(T')}_{\BMO} \le C \eps$. Provided $\eps$ is small enough, applying the existence and uniqueness result Theorem \ref{thm:existence} to initial data $a = u(T')$ implies $\norm{u(T' + \cdot)}_{X} \le C\eps$, thus the desired asymptotic behaviour \eqref{eq:asym} of $u$.

First we show the persistence of $g(t)$ in $\BMO$ for small initial data $g_0$.
\begin{lemma}\label{la:g}
Let $\norm{g_0}_{\BMO}< \eps$ for some $ \eps < \min \set{1,\rho}$. Then the global solution $g \in X$ of \eqref{eq:main} to initial data $g_0$ satisfies $g(t) \in \BMO$ for any $t \in (0,\infty)$ with
\[
\norm{g(t)}_{\BMO} < C \eps
\]
for a constant $C>0$ independent of $t$.
\end{lemma}

\begin{proof}
First we note that for $F(\xi)=\pm\abs{\xi}^2 \xi$, the solution is characterized by
\[
g(t) = S(t) g_0 \pm \Psi(g,g,g)(t).
\]
It follows by Lemma \ref{la:Lpersist} and Proposition \ref{prop:existence} that
\[
\norm{g(t)}_{\BMO} \le C\bra*{ \norm{g_0}_{\BMO}+  \norm{g}_{X}^3} \le C \norm{g_0}_{\BMO}.
\]
Hence the claim for some constant $C>0$ independent of $t$.
\end{proof}

In the following we prove that there exists a time $T'>0$ at which $\norm{f(T')}_{\BMO} \le \eps$, under the condition (a) or (b) of Theorem \ref{thm:stability} separately.

\subsubsection{The coercive case}

For the coercive nonlinearity $F(\xi)=\abs{\xi}^2\xi$, we first show that the perturbation $f=u-g$ fulfils an energy inequality. Then we derive its decay property.

We first prove that $f=u-g$ possesses following regularity property:
\begin{lemma}\label{la:f_localregularity}
$f \in L_T^p$ for any $T \in(0,\infty)$.
\end{lemma}
\begin{proof}
We note that $f$ satisfies the equation 
\begin{equation}\label{eq:f}
f(t)=S(t)f_0 - \mathcal{L}_{u,g}f(t)+\Psi(f,f,f)(t).
\end{equation}
This equation can be rewritten as
\[
f(t)=S(t)f_0 - \mathcal{L} f(t)
\]
where $\mathcal{L}=  \mathcal{L}_{u,g}-\frac{1}{3} \mathcal{L}_{f,f}$ for $u,g,f \in X_T$. Since $f_0 \in C_0^\infty(\R^n)$, we have $S(t)f_0 \in L^p_T$. Due to Lemma  \ref{la:Nshift} and \ref{la:LcontinuityL}, we can cover interval $(0,T)$ by finitely many overlapping intervals $I_j = (t_j, t_j')$ with $t_0=0$, $t_J'=T$ and $t_j < t'_{j-1}$ for any $1 \le j\le J$, such that \[ \norm{\mathcal{L}}_{L^p_{t_j' - t_j}} \le \frac{1}{2}.\] Note that since $T<\infty$, we do not need the asymptotic assumption \eqref{eq:asym_ab_inf}. Then following a similar argument as in Lemma \ref{la:spectrum}, we can prove that $I+\mathcal{L}$ is invertible on $L_T^p$. Therefore $f(t) = (I+\mathcal{L})^{-1} S(t)f_0$ holds in $L_T^p$.
\end{proof}
This regularity property of $f$ is local in time. Now we exploit the coercivity of the nonlinearity to derive global $L^2$-regularity of $f$. We infer from \eqref{eq:f} that $f$ solves the following equation
\[
\begin{aligned}
\partial_t f + (-\lap)^2 f &= \nabla \cdot \bra*{\abs{\nabla f}^2 \nabla f} + \abs{\nabla g}^2 \lap f + \abs{\nabla f}^2 \lap g
+2(\nabla g \cdot \nabla f) \lap g +2(\nabla g \cdot \nabla f) \lap f\\
& \quad + 2 \nabla g \cdot \nabla^2 g \cdot \nabla f + 2 \nabla g \cdot \nabla^2 f \cdot \nabla g + 2 \nabla g \cdot \nabla^2 f \cdot \nabla f\\
& \quad + 2 \nabla f \cdot \nabla^2 g \cdot \nabla g + 2 \nabla f \cdot \nabla^2 g \cdot \nabla f+ 2 \nabla f \cdot \nabla^2 f \cdot \nabla g.
\end{aligned}
\]
Due to Lemma \ref{la:f_localregularity}, we can multiply this equation by $f$ and integrate it over $\R^n \times (T,T')$ for any $0<T<T'<+\infty$. Regularity in time follows from the properties of semigroup and spatial regularity. Then applying partial integration and H\"older's inequality, we infer that
\[
\begin{aligned}
&\frac{1}{2}  \norm{f(T')}_{L^2}^2 + \int_T^{T'} \norm{\lap f(t)}_{L^2}^2 \dif t
\le \frac{1}{2}  \norm{f(T)}_{L^2}^2+ \int_T^{T'} \bigg[ -\norm{\nabla f(t)}_{L^4}^4 - 5 \norm{\nabla g(t) \cdot \nabla f(t)}_{L^2}^2 \\
&\qquad +N_{2,\infty}(g) t^{-\frac{1}{2}} \norm{\nabla f(t)}_{L^4}^2 \norm{f(t)}_{L^2}
 + 2N_{1,\infty}(g)N_{1,\infty}(f) t^{-\frac{1}{2}} \norm{\lap f(t)}_{L^2}\norm{f(t)}_{L^2} \\
 &\qquad+ 2N_{1,\infty}^2(g) t^{-\frac{1}{2}} \norm{\lap f(t)}_{L^2}\norm{f(t)}_{L^2} \bigg] \dif t.
\end{aligned}
\]
It follows from Young's inequality and the smallness of $g$ that
\begin{equation}\label{eq:f_asym}
\begin{aligned}
&\norm{f(T')}_{L^2}^2 + \int_T^{T'}  \norm{\lap f(t)}_{L^2}^2 \dif t\\
&\le  \norm{f(T)}_{L^2}^2 
+\int_T^{T'} \pra*{8 N_{1,\infty}^2(g) N_{1,\infty}^2(f)+ 8N_{1,\infty}^4(g) +\frac{1}{4} N_{2,\infty}^2(g) } t^{-1} \norm{f(t)}_{L^2}^2 \dif t\\
&\le \norm{f(T)}_{L^2}^2 + C_u\, \eps^2 \sup_{T \le t \le T'} \norm{f(t)}_{L^2}^2 \ln \frac{T'}{T},
\end{aligned}
\end{equation}
where $C_u>0$ is a constant depending on $u_0$. Setting $t_k = e^k$, $k \in \N$, we obtain
\[
\norm{f(T')}_{L^2}^2 + \int_{t_k}^{T'} \norm{\lap f(t)}_{L^2}^2 \dif t \le \norm{f(t_k)}_{L^2}^2 + C_u\, \eps^2 \sup_{t_k \le t \le T'} \norm{f(t)}_{L^2}^2 \quad \forall T' \in [t_k,t_{k+1}]
\]
and
\[
\sup_{t_k \le t \le T'} \norm{f(t)}_{L^2}^2 \le 2 \norm{f(t_k)}_{L^2}^2
\]
if $\eps \le 1/\sqrt{C_u}$, hence
\[
\norm{f(t_{k+1})}_{L^2}^2 \le (1+2C_u \eps^2) \norm{f(t_k)}_{L^2}^2.
\]
This implies
\[
\norm{f(t)}_{L^2}^2 \le Ct^\beta, \qquad\forall t \ge 1 \quad \text{where } \beta = \ln(1+2C_u \eps^2),
\]
where $\eps$ is the upper bound of $\norm{g(t)}_{\BMO}$. Since $f = u-g \in X$, we have
\[
\norm{\nabla^2 f(t)}_{L^\infty} \le Ct^{-\frac{1}{2}}
\]
and by the Gagliardo–Nirenberg inequality 
\[
\norm{\nabla f(t)}_{L^n} \le C\norm{\nabla^2 f(t)}_{L^\infty}^{\frac{n}{n+4}} \norm{f(t)}_{L^2}^{\frac{4}{n+4}}
\le Ct^{\frac{4\beta-n}{2(n+4)}} \to 0 \quad\text{as}\quad t \to \infty
\]
for any $n \ge 4$, provided $\eps \le \sqrt{(e^{n/4}-1)/2C_u}$. 

For $n=2,3$, we note that it follows from \eqref{eq:f_asym}
\[
\int_{t_k}^{t_{k+1}} \norm{\lap f(t)}_{L^2}^2 \dif t \le \norm{f(t_k)}_{L^2}^2 + C_u\, \eps^2 \sup_{t_k \le t \le T'} \norm{f(t)}_{L^2}^2
\]
and thus
\[
\int_{1}^{T}  \norm{\lap f(t)}_{L^2}^2 \dif t \le CT^\beta \quad \forall T>1.
\]
Then by Gagliardo–Nirenberg inequality, we obtain
\[
\begin{aligned}
\frac{1}{T} \int_{1}^{T}  \norm{\nabla f(t)}_{L^n}^{\frac{8}{n}} \dif t & \le C  \bra*{\frac{1}{T} \int_{1}^{T}  \norm{\lap f(t)}_{L^2}^{2} \dif t }\bra*{ \sup_{t \in [0,T]} \norm{f(t)}_{L^2}^{\frac{4-n}{n} \frac{8}{n}}}\\
&  \le CT^{3\beta-1} \to 0 \quad\text{as}\quad T \to \infty
\end{aligned}
\]
provided $\eps \le \sqrt{(e^{\frac{1}{3}}-1)/2C_u}$.

Therefore we have $\lim\limits_{t \to \infty} \norm{f(t)}_{\BMO}= 0$ for any $n \ge 2$, due to the Poinc\'are inequality $\norm{a}_{\BMO(\R^n)} \lesssim \norm{\nabla a}_{L^n(\R^n)}$.

\subsubsection{The non-coercive case}
Now we consider the nonlinearity $F(\xi)= -\abs{\xi}^2 \xi$, in which case $f = u-g$ does not satisfy an energy-type inequality as above. But if $u_0$ is small enough, the asymptotic property of $f$ still holds.

Suppose $u_0 \in \VMO$ satisfies $\norm{u_0}_{\BMO} < \rho$ and $\norm{g_0}_{\BMO} < \eps < \rho$ where $\rho>0$ is the constant in Theorem \ref{thm:existence}. Then the corresponding solutions $u,g \in X$ satisfy $\norm{u}_{X} < 2C_1 \rho$ and $\norm{g}_{X} < 2C_1 \eps$. In the following, we deduce the global $L^p$-regularity of $f$. 
\begin{lemma}\label{la:f_regularity}
Under the above conditions with possibly reduced $\rho$, we have $f \in L_T^p$ for $T =+\infty$.
\end{lemma}
\begin{proof}
Since $\norm{u_0}_{\BMO} < \rho$, there exists a sequence $\bra{u_k}_{k \in \N}$ obtained by the fixed point iteration \eqref{eq:FP}, which converges to $u$ in $X_T$ and satisfies $\norm{u_k}_{X_T} < 2C_1 \rho$.
We define 
\[
f_1 = S(t)f_0, \qquad f_{k+1} = u_{k+1} - g \quad\text{for } k \in \N.
\]
Using the notation of the trilinear operator $\Psi$, we can reformulate $f_k$ as
\[
f_{k+1} = u_{k+1}-g = f_1 -\mathcal{L}_{u_k,g}f_k - \Psi(f_k,f_k,f_k).
\]
Since $f_0 \in C_0^\infty$, it follows from \eqref{eq:St} that
 for $t>0$
\[
\norm{\nabla S(t) f_0}_{L^p} \le C \norm{\nabla f_0}_{L^p}
\]
and
\[
\norm{\nabla^2 S(t) f_0}_{L^p} \le C t^{-\frac{1}{4}}\norm{\nabla f_0}_{L^p}, 
\]
hence $f_1 \in L_T^p$. Letting $\mu:=\norm{f_1}_{L_T^p}$, then if $\norm{f_k}_{L_T^p} < 2\mu$ for some $k \in \N$, Lemma \ref{la:LcontinuityL} implies
\[
\begin{aligned}
\norm{f_{k+1}}_{L_T^p} & \le \norm{f_1}_{L_T^p} + \norm{\mathcal{L}_{u_k,g}f_k}_{L_T^p} + \norm{ \Psi(f_k,f_k,f_k)}_{L_T^p} \\
& \le \norm{f_1}_{L_T^p} + \bra*{C_4 \norm{u_k}_{X_T} \norm{g}_{X_T} +C_3\norm{f_k}_{X_T}^2 } \norm{f_k}_{L_T^p} \\
& \le \mu + \bra*{4C_1^2 C_4 \rho \eps + 4C_1^2 C_3 (\rho + \eps)^2} \mu < 2\mu
\end{aligned}
\]
provided $\rho^2 < \dfrac{1}{4C_1^2 (C_4 + 4C_3)}$. Therefore the sequence $\bra{f_k}_{k \in \N}$ is uniformly bounded in $L_T^p$, hence its limit $f \in X_T \cap L_T^p$.
\end{proof}

According to the definition of $L_T^n$, we know that $\lim\limits_{t \to \infty} \norm{\nabla f(t)}_{L^n}= 0$, hence $\lim\limits_{t \to \infty} \norm{f(t)}_{\BMO}= 0$.

\begin{remark}
Applying Lemma \ref{la:f_regularity} to $g=0$, we conclude the following result:
If additionally the initial data $u_0 \in C_0^\infty(\R^n)$, then the solution $u$ obtained in Theorem \ref{thm:existence} is in $L_T^p$ for any $p \in [0,\infty]$.
\end{remark}

Now we put everything together to prove Theorem \ref{thm:stability}:

\textbf{Proof of Theorem \ref{thm:stability}. } Under the condition (a) or (b), for an arbitrary $\eps>0$, there exists a time $T'>0$ at which $\norm{f(T')}_{\BMO} \le \eps$. Together with Lemma \ref{la:g}, this yields  $\norm{u(T')}_{\BMO} \le C\eps$. Hence by the well-posedness theory to small initial data established in Theorem \ref{thm:existence}, we infer that $ \norm{u (T'+ \cdot)}_X \le C \eps$, i.e., the large time behaviour of global solution in Theorem \ref{thm:stability}(i).

Then we can invoke Lemma \ref{la:spectrum} with $a=b=u$ and conclude that $I+\mathcal{L}_{u,u}$ is invertible on $X$. According to Lemma \ref{la:NL}, equation \eqref{eq:w1} is solvable for small enough perturbation $w_0 = u_0 - v_0$, which implies the stability of global solutions and completes the proof of Theorem \ref{thm:stability}(ii).


\nocite{*}
\bibliographystyle{abbrv}
\bibliography{references}

\end{document}